\documentclass[11pt]{amsart}
\usepackage{physics}
\usepackage{amssymb}
\usepackage{xcolor}

\usepackage[
  hmarginratio={1:1},     % equal left and right margins
]{geometry}

\usepackage{amsthm}

\usepackage{graphicx} % support the \includegraphics command and options

\theoremstyle{plain}
\newtheorem{theorem}{Theorem}[section]
\newtheorem{corollary}[theorem]{Corollary}
\newtheorem{proposition}[theorem]{Proposition}

\newtheorem{lemma}[theorem]{Lemma}

\theoremstyle{definition}
\newtheorem{definition}[theorem]{Definition}

\newtheorem*{remark*}{Remark}
\numberwithin{equation}{section}

\makeatletter
\@namedef{subjclassname@2020}{%
  \textup{2020} Mathematics Subject Classification}
\makeatother

\begin{document}
\title[Information geometry and $\alpha$-parallel prior of BLD]{Information geometry and $\alpha$-parallel prior of the beta-logistic distribution}
\author{Lin Jiu}
\address{Zu Chongzhi Center for Mathematics and Computational Sciences, Duke Kunshan University, Kunshan, Jiangsu Province, 215316, P.~R.~China}
\email{lin.jiu@dukekunshan.edu.cn}
\author{Linyu Peng}
\address{Department of Mechanical Engineering, Keio University, 
Yokohama 223-8522, Japan}
\email{l.peng@mech.keio.ac.jp}
%\author{Jiongyi Wang}
%\address{Department of Mechanical Engineering, Keio University, 
%Yokohama 223-8522, Japan}
%\email{jiongyiwang@keio.jp}

\begin{abstract}
%In the current work, we study the dual geometric structures of the beta-logistic distribution. As one generalization of the hyperbolic secant distribution, it 

The hyperbolic secant distribution has several generalizations with applications in finance. In this study, we explore the dual geometric structure of one such generalization, namely the beta-logistic distribution. Recent findings also interpret Bernoulli and Euler polynomials as moments of specific random variables, treating them as special cases within the framework of the beta-logistic distribution.  The current study also uncovers that  the beta-logistic distribution admits an $\alpha$-parallel prior for any real  number $\alpha$, that  has the potential for application in geometric statistical inference.
\end{abstract}

\maketitle

\section{Introduction}

The originality of the \emph{hyperbolic secant distribution} (HSD)  can date back to Fisher \cite{RAFisher}, with a simple probability density function on $\mathbb{R}$
\begin{equation}
\frac{\sech x}{\pi}=\frac{1}{\pi\cosh x}=\frac{2}{\pi(e^x+e^{-x})}. 
\end{equation}
Here, the factor $1/\pi$ is for normalization. A formal construction begins with two independent Gaussian distributions, say $Y_1$ and $Y_2$. It is well-known that $Y_1/Y_2$ admits a Cauchy distribution, with density $1/(\pi(x^2+1))$. Then, it is not hard to check that $\log(|Y_1/Y_2|)$ satisfies a HSD. 

Generalizations of HSD have been well studied, with applications in finance; see, e.g., \cite{SGHS} for a comprehensive collection. For example, Perks \cite{Perks} (see also \cite[Chpt.~2]{SGHS}) introduced the rational extension in $e^{-x}$, of the HSD, with the density function of the form
\begin{equation}
f(x)=\frac{a_0+a_1 e^{-x}+a_2 e^{-2 x}+\cdots+a_m e^{-m x}}{b_0+b_1 e^{-x}+b_2 e^{-2 x}+\cdots+b_n e^{-n x}},
\end{equation}
where proper parameters $a_0,a_1,\ldots,a_m,b_0,b_1,\ldots,b_n$ make $f(x)$ a probability density function.  
Baten \cite{Baten} (see also \cite[Chpt.~3]{SGHS}),  originates to consider the sum of i.i.d.~hyperbolic secant distributions. Then, Harkness and Harkness \cite{GHSD}, on the other hand, considered the sum of i.i.d., whose moment generating function is of the form $\sech^\rho (\alpha x)$. This leads to the \emph{skew generalized hyperbolic secant distribution}, with density
\begin{equation}
f(x;\rho,h)=\frac{2^{\rho-2}}{\pi \Gamma(\rho)} \cdot\left|\Gamma\left(\frac{\rho}{2}+\mathbf{i}\frac{x}{2}\right)\right|^2 \exp \left(h x+\rho \ln \cos (h)\right),
\end{equation}
where $\Gamma$ is the gamma function and $\mathbf{i}$ is the imaginary unit. The parameters are subject to $\rho>0$ and $|h|< \pi/2$. The special case $\rho\mapsto\rho/2$ and $x\mapsto x/2$ leads to the Meixner distribution. 

In this work, we focus on  another generalized hyperbolic secant distribution. Instead of having the moment generating function as the power of a hyperbolic secant function, we study the case that the density function being propositional to the power of the hyperbolic secant function. More precisely, we consider the family $X_{\theta}=X_{(\theta^1,\theta^2)}$ with the density function
\begin{equation}\label{eq:Density}
    p(x;\theta^{1},\theta^{2})=\frac{\sech^{\theta^{1}}(x)\cdot \exp({\theta^2 x})}{\int_{\mathbb{R}}\sech^{\theta^{1}}(x)\cdot \exp({\theta^2 x})\operatorname{d}\!x},\quad x\in\mathbb{R},
\end{equation}
where two parameters $\theta^1$ and $\theta^2$ are introduced, with $\theta^1\pm\theta^2>0$. This is, up to simple factors, the \emph{exponential generalized beta of the second kind (EGB2) distribution} or the \emph{beta-logistic distribution}, whose density function is given by, for $\beta_1,\beta_2>0$, 
\begin{equation}
f\left(x ; \beta_1, \beta_2\right)=\frac{1}{B\left(\beta_1, \beta_2\right)}\cdot \frac{\exp \left(\beta_1 x\right)}{(1+\exp (x))^{\beta_1+\beta_2}},
\end{equation}
involving the beta function $B(x,y)$. See, e.g., \cite[Sect.~4.4]{SGHS}. 

A second motivation on studying the family $X_\theta$ lies in  the probabilistic interpretations of Bernoulli and Euler polynomials, both of which are important polynomials appearing in various areas of mathematics, e.g., number theory, combinatorics, numerical analysis, etc. They are defined as follows. 
\begin{definition}
    The \emph{Bernoulli polynomials} $B_n(x)$ and \emph{Euler polynomials} $E_n(x)$ are defined by their exponential generating functions as follows
    \begin{equation}
    \sum_{n=0}^{\infty}B_{n}\left(x\right)\frac{t^{n}}{n!}=\frac{te^{xt}}{e^{t}-1}\qquad\text{and}\qquad
    \sum_{n=0}^{\infty}E_{n}\left(x\right)\frac{t^{n}}{n!}=\frac{2e^{xt}}{1+e^{t}}.
    \end{equation}
    In particular, $B_n=B_n(0)$ and $E_n=2^n E_n(1/2)$ are the \emph{Bernoulli numbers} and \emph{Euler numbers}, respectively.  
\end{definition}
The umbral calculus (see, e.g., \cite{Umbral}) allows us to apply the Bernoulli umbra $\mathcal{B}$ and Euler umbra $\mathcal{E}$, with only simple evaluations: 
\begin{equation}
eval\left[(\mathcal{B}+x)^n\right]=B_n(x)\qquad\text{and}\qquad
eval\left[(\mathcal{E}+x)^n\right]=E_n(x).
\end{equation}
The probabilistic interpretations involve the hyperbolic secant function. More precisely, we can let $L_B$ be the random variable subject to the density function $\pi\sech^2(\pi x)/2$ and $L_E$ be the one with density function $\sech(\pi t)$, then 
\begin{equation}\label{eq:RV}
    \mathcal{B}=\mathbf{i}L_B-\frac{1}{2}\qquad\text{and}\qquad\mathcal{E}=\mathbf{i}L_E-\frac{1}{2}.
\end{equation}
Namely, (see, e.g., \cite[Thm.~2.3]{Zagier1})
\begin{equation}
\begin{aligned}
B_{n}(x)=eval\left[(\mathcal{B}+x)^n\right]=\mathbb{E}\left[(\mathcal{B}+x)^n\right]=\frac{\pi}{2}\int_{\mathbb{R}}\left(x+\mathbf{i}t-\frac{1}{2}\right)^{n}\sech^{2}(\pi t)\operatorname{d}\!t,
\end{aligned}
\end{equation}
and (see, e.g., \cite[Prop.~2.1]{Euler})
\begin{equation}
\begin{aligned}
E_{n}(x)=eval\left[(\mathcal{E}+x)^n\right]=\mathbb{E}\left[(\mathcal{E}+x)^n\right]=\int_{\mathbb{R}}\left(x+\mathbf{i}t-\frac{1}{2}\right)^{n}\sech(\pi t)\operatorname{d}\!t.
\end{aligned}
\end{equation}
%where $i^2=-1$. 

As both important polynomials are moments of certain random variables, with powers of hyperbolic secant in the density functions, studies on the family $X_\theta$ with density \eqref{eq:Density} may potentially reveal more hidden connections between $B_n(x)$ and $E_n(x)$, as well as general properties of the whole family. 

This paper is constructed as follows. In Section \ref{sec:pre}, we introduce basics in information geometry and several important special functions including Riemann and Hurwitz zeta functions, gamma and beta functions, and the polygamma function, in order to make this paper self-contained. Then, in Section \ref{sec:geometry}, we calculate the geometry structures of the beta-logistic distribution, such as Fisher information metric, $\alpha$-connections, $\alpha$- curvatures, and the geodesic equations. At the end of this section, as examples, we apply the formulas to the cases related to Bernoulli polynomials and Euler polynomials. Finally, in Section \ref{sec:conclustion}, we summarize the results and discuss some remarks on further work.

\section{Preliminaries}\label{sec:pre}
In this section, we will introduce some basics in differential geometry, information geometry, special functions, and etc., in order to make this paper self-contained. References are listed at the very beginning of each subsection. 
\subsection{Information Geometry}
Fundamentals of information geometry can be found, e.g., in \cite{Amari} and \cite{Sun}.
\begin{definition}
   For a probability density function $p(x;\theta)$, where the parameters are $\theta=(\theta^1,\theta^2,\ldots,\theta^n)\in \Theta\subset \mathbb{R}^n$, 
   %we define the function
%    \begin{equation}
%    l(x,\theta)=\ln p(x,\theta).
%    \end{equation}
 we call 
 \begin{equation}
 M=\left\{ p(x;\theta)\mid(\theta^1,\theta^2,\ldots,\theta^n)\in \Theta\subset \mathbb{R}^n\right\} 
 \end{equation}
  \emph{an $n$-dimensional statistical manifold} with local coordinates $(\theta^1,\theta^2,\ldots,\theta^n)$.
\end{definition}

In the following, we assume all probability density functions satisfy the regularity conditions for computational convenience; for details of  the regularity conditions, the reader may refer to \cite{Amari}.

\begin{definition}
   The manifold $M$ becomes Riemannian by introducing the \emph{Fisher information metric} with local matrix expression $(g_{ij})$ defined by 
    \begin{equation}
    g_{ij}=\mathbb{E}\left[\partial_i l  \quad\!\!\! \partial_j l\right],\quad i,j=1,2,\ldots,n,
    \end{equation}
    where the function $l$ stands for $l(x;\theta)=\ln  p(x;\theta)$ and 
        \begin{equation}
    \partial_i l=\frac{\partial l(x;\theta)}{\partial \theta^i}.
    \end{equation}
    \end{definition}
    
    We denote the determinant of the Fisher information matrix as $\left|G\right|:=\det(g_{ij})$ and the inverse as $(g^{ij})=(g_{ij})^{-1}$.  For $i,j,k=1,2,\ldots,n$, the Riemannian connection coefficients $\Gamma_{ijk}$ with respect to the Fisher information metric are given by
    \begin{equation}
    \Gamma_{ijk}=\frac{1}{2}\left(\partial_ig_{jk}+\partial_jg_{ki}-\partial_kg_{ij}\right).
    \end{equation}

\begin{definition}
   Statistical manifolds admit a one-parameter family of dual connections, denoted by \emph{$\alpha$-connections}, with  connection coefficients  defined by
    \begin{equation}
    \Gamma_{i j k}^{(\alpha)}=\Gamma_{i j k}-\frac{\alpha}{2} T_{i j k},
    \end{equation}
    where 
    \begin{equation}\label{eq:Tijk}
        T_{i j k}=E\left[\partial_i l \quad\!\!\!  \partial_j l  \quad\!\!\! \partial_k l\right].
    \end{equation}

\end{definition}
%\begin{definition}
    Standard computation gives the corresponding $\alpha$-curvatures. In the $\theta$ coordinate system, using the Einstein summation convention, 
    \begin{enumerate}
        \item the \emph{$\alpha$-curvature tensors} $R_{i j k l}^{(\alpha)}$ are given by
        \begin{equation}
        R_{i j k l}^{(\alpha)}=\left(\partial_j \Gamma_{i k}^{s(\alpha)}-\partial_i \Gamma_{j k}^{s(\alpha)}\right) g_{s l}+\left(\Gamma_{j t l}^{(\alpha)} \Gamma_{i k}^{t(\alpha)}-\Gamma_{i t l}^{(\alpha)} \Gamma_{j k}^{t(\alpha)}\right),\quad i,j,k,l=1,2,\ldots,n,
        \end{equation}
        where
        \begin{equation}
        \Gamma_{i j}^{k(\alpha)}=\Gamma_{i j s}^{(\alpha)} g^{s k};
        \end{equation}
        \item the \emph{$\alpha$-Ricci curvatures} $R_{i k}^{(\alpha)}$ (in components) are 
        \begin{equation}
        R_{i k}^{(\alpha)}=R_{i j k l}^{(\alpha)} g^{j l};
        \end{equation}
        \item the \emph{$\alpha$-scalar curvature} is 
        \begin{equation}
        R^{(\alpha)}=R^{(\alpha)}_{ij}g^{ij}.
        \end{equation}
    \end{enumerate}
    Finally, when $n=2$, the $\alpha$-sectional curvature coincides with the \emph{$\alpha$-Gaussian curvature}  given by $K^{(\alpha)}=R^{(\alpha)}_{1212}/|G|$. When $\alpha=0$, the connection becomes self-dual, which corresponds to the Riemannian structure. 
%\end{definition}

%\begin{remark}
%    In all the definition above, the Einstein summation rule applies, as both upper and lower indices are summed over. For example, 
%    \begin{equation}
%    c_jx^j=\sum_{j=1}^n c_j x^j.
%    \end{equation}
%    Hence, 
%    \begin{equation}
%    \Gamma_{i j}^{k(\alpha)}=\Gamma_{i j s}^{(\alpha)} g^{s k}=\sum_{s=1}^n\Gamma_{i j s}^{(\alpha)} g^{s k}.
%    \end{equation}
%    For all the $\alpha$-connections, $\alpha$-tensors, and $\alpha$-curvatures, when $\alpha=0$, one recovers the Riemannian connections, tensors and curvatures. 
%\end{remark}

\begin{definition}
    The \emph{geodesic equations} of the manifold $(M,g)$ with coordinate $\theta=(\theta^1,\theta^2,\ldots,\theta^n)$ are %characterized by
    \begin{equation}
        \frac{\operatorname{d}^2\!\theta^k}{\operatorname{d}\!t^2}+\Gamma_{ij}^k\frac{\operatorname{d}\!\theta^i}{\operatorname{d}\!t}\frac{\operatorname{d}\!\theta^j}{\operatorname{d}\!t}=0.
    \end{equation}
\end{definition}
%The relative geodesic spread is characterised by the Jacobi field $J^i(\theta)\partial_i$, that satisfies the
%Jacobi--Levi-Civita equations (see, e.g., \cite{doC1992,Pengetal2011})
%\begin{equation}\label{eq:JLC}
%\frac{D^2\! J^i}{D t^2} +R_{kml}^i\frac{\operatorname{d}\!\theta^k}{\operatorname{d}\!t}\frac{\operatorname{d}\!\theta^l}{\operatorname{d}\!t}J^m=0,
%\end{equation}
% where $\theta(t)$ is solution of  the geodesic equations, and $D/Dt$ is the covariant derivative yielding
% \begin{equation}
% \begin{aligned}
% \frac{D^2\! J^i}{Dt^2}=\frac{\operatorname{d}^2\!J^i}{\operatorname{d}\!t^2}+2\Gamma_{kj}^i\frac{\operatorname{d}\!\theta^j}{\operatorname{d}\!t}\frac{\operatorname{d}\!J^k}{\operatorname{d}\!t}+\Gamma^i_{kj}\frac{\operatorname{d}^2\!\theta^j}{\operatorname{d}\!t^2}J^k +\Gamma^i_{kj,l}\frac{\operatorname{d}\!\theta^j}{\operatorname{d}\!t}\frac{\operatorname{d}\!\theta^l}{\operatorname{d}\!t} J^k+\Gamma^i_{kj}\Gamma^k_{ml}\frac{\operatorname{d}\!\theta^j}{\operatorname{d}\!t}\frac{\operatorname{d}\!\theta^l}{\operatorname{d}\!t}J^m.
% \end{aligned}
% \end{equation}
% The convergence/divergence order of geodesic spread can then be described by
% \begin{equation}
% J^2=g_{ij}J^iJ^j.
% \end{equation}

\subsection{Several special functions}
In this subsection, we provide the definitions, important properties of the gamma, beta, zeta, and polygamma functions. Most of them can be found, e.g., \cite[Chpts.~5 and 25]{DLMF}. 

\begin{definition}
The \emph{gamma function} $\Gamma(z)$ is defined to be the analytic continuation of the integral represented function
\begin{equation}
\Gamma(z)=\int_0^\infty e^{-t}t^{z-1}\operatorname{d}\!t, 
\end{equation}
on the right-half plane $\Re(z)>0$. It is a meromorphic function on $\mathbb{C}$ with no zeros, and simple poles of residue $(-1)^n/n!$ at nonpositive integer points $z=-n$, for $n=0,1,2,\ldots$.

The \emph{polygamma function of order $m$} is the $(m+1)$th derivative of the logarithm of $\Gamma(z)$, i.e., 
\begin{equation}
\psi^{(m)}=\frac{\operatorname{d}^m}{\operatorname{d}\!z^m}\psi(z)=\frac{\operatorname{d}^{m+1}}{\operatorname{d}\!z^{m+1}}\ln\Gamma(z),
\end{equation}
where $\psi(z)=\psi^{(0)}(z)=\Gamma'(z)/\Gamma(z)$.

The \emph{beta function $B(x,y)$} is the analytic continuation of the the integral represented function
\begin{equation}
B(x,y)=\int_{0}^{1}t^{x-1}(1-t)^{y-1}\operatorname{d}\!t,
\end{equation}
for $\Re(x)>0,\Re(y)>0$. 

The \emph{Riemann zeta-function} is the analytic continuation of the series expansion
\begin{equation}
\zeta(s)=\sum_{n=1}^\infty \frac{1}{n^s},
\end{equation}
and the \emph{Hurwitz zeta-function} is similarly defined, with an extra parameter $a\neq 0,-1,-2,\ldots$, 
\begin{equation}
\zeta(s,a)=\sum_{n=0}^\infty\frac{1}{(n+a)^s}.
\end{equation}
\end{definition}

\begin{lemma}\label{lem:BG}
The following identities are well known:
\begin{equation}
\begin{aligned}
    B(x,y)&=\frac{\Gamma(x)\Gamma(y)}{\Gamma(x+y)},\\
    \psi^{(m)}(z)&=(-1)^{m+1}m!\zeta(m+1,z).
\end{aligned}
\end{equation}
For a positive integer $k$, 
\begin{equation}
\psi^{(m)}(k)=(-1)^{m+1}m!\left(\zeta(m+1)-H_{k-1}^{(m+1)}\right),
\end{equation}
where the \emph{generalized harmonic number} is defined as 
\begin{equation}
H_n^{(r)}:=1+\frac{1}{2^r}+\frac{1}{3^r}+\cdots+\frac{1}{n^r}.
\end{equation}
\end{lemma}

Finally, if we recall the Bateman's tables \cite[Entry 1.9.5]{FourierHypersecant}
that for $x\in\mathbb{R}$, 
\begin{equation}
\int_{0}^{\infty}\sech^{\rho}\left(\alpha t\right)\cos\left(xt\right)\operatorname{d}\!t=\frac{2^{\rho-2}}{\alpha}B\left(\frac{\rho}{2}+\frac{\mathbf{i}x}{2\alpha},\frac{\rho}{2}-\frac{\mathbf{i}x}{2\alpha}\right),
\end{equation}
then, by $e^{\mathbf{i}xt}=\cos(xt)+\mathbf{i}\sin(xt)$ and the parities of cosine and sine functions, we have 
\begin{equation}
\int_{\mathbb{R}}\sech^\rho te^{\mathbf{i}tx} \operatorname{d}\!t=2^{\rho-1}B\left(\frac{\rho+\mathbf{i}x}{2},\frac{\rho-\mathbf{i}x}{2}\right),
\end{equation}
which implies the density in \eqref{eq:Density} should be
\begin{equation}\label{eq:pdf}
    p(x;\theta^1,\theta^2)=\frac{2^{1-\theta^1}\sech^{\theta^1}(x) \exp(\theta^2 x)}{B\left(\frac{\theta^1-\theta^2}{2},\frac{\theta^1+\theta^2}{2}\right)},
\end{equation}
a beta-logistic distribution due to normalization.

\section{Geometric Structure of the Beta-Logistic Distribution}\label{sec:geometry}
Now, we consider the family $X_\theta$ and study the corresponding statistical manifold 
\begin{equation}\label{eq:stma}
M=\{p(x;\theta^1,\theta^2), x\in\mathbb{R}\mid \theta^1\pm\theta^2>0\},
\end{equation}
 where the density $p(x;\theta^1,\theta^2)$ is given by \eqref{eq:pdf}. Notice that 
\begin{equation}\label{eq:lll}
    l(x;\theta^1,\theta^2)=\ln p(x;\theta^1,\theta^2)=\theta^{1}\ln\sech(x)+\theta^{2}x-\phi(\theta^{1},\theta^{2}),
\end{equation}
where the potential function is 
\begin{equation}\label{eq:pot}
\phi(\theta^{1},\theta^{2})=\ln B\left(\frac{\theta^{1}-\theta^{2}}{2},\frac{\theta^{1}+\theta^{2}}{2}\right)+(\ln2)\left(\theta^{1}-1\right).
\end{equation}
%Note that, since $\phi$ does not depend on $x$ and $l-\phi$ is linear in both $\theta^1$ and $\theta^2$, any partial derivatives higher than $1$ of $l$ only depends on the potential function $\phi$. 
%{\color{red} (3.1) implies this is an exponential model. }

\subsection{Dual geometric structures}
Eq. \eqref{eq:lll} implies that this gives an exponential family with $(\theta^1,\theta^2)$ the natural coordinates (see, e.g., \cite{Amari}), whose dual geometric structure can be totally determined by the potential function \eqref{eq:pot}.
 
For example, the Fisher information matrix  
\begin{equation}
g_{ij}=-\mathbb{E}\left[\partial_i\partial_j l\right]=\partial_i\partial_j\phi %\frac{\partial^2\phi}{\partial\theta^i\partial\theta^j}=
\end{equation}
is given by the Hessian of the potential function. 
Furthermore,
\begin{equation}
\Gamma_{i j k}^{(\alpha)}=\Gamma_{i j k}-\frac{\alpha}{2} T_{i j k}=\frac{1-\alpha}{2}\partial_i\partial_j\partial_k\phi,
\end{equation}
and
\begin{equation}\label{eq:CurvatureTensor}
    R_{ijkl}^{(\alpha)}=\frac{1-\alpha^{2}}{4}g^{mn}(T_{kmi}T_{jln}-T_{kmj}T_{iln}).
\end{equation}

%Hence, all depend on the partial derivatives of the potential function $\phi$. We first begin with calculating the Fisher matrix. 
\begin{proposition}
    The Fisher matrix $G=(g_{ij})_{2\times2}$ of the manifold  $M$ defined in \eqref{eq:stma} is given by 
    \begin{equation}\label{eq:fis}
    G=\left(\begin{matrix}\frac{1}{4}\psi'\left(\frac{\theta^{1}+\theta^{2}}{2}\right)+\frac{1}{4}\psi'\left(\frac{\theta^{1}-\theta^{2}}{2}\right)-\psi'(\theta^{1}) & \frac{1}{4}\psi'\left(\frac{\theta^{1}+\theta^{2}}{2}\right)-\frac{1}{4}\psi'\left(\frac{\theta^{1}-\theta^{2}}{2}\right)\\
    \\
    \frac{1}{4}\psi'\left(\frac{\theta^{1}+\theta^{2}}{2}\right)-\frac{1}{4}\psi'\left(\frac{\theta^{1}-\theta^{2}}{2}\right) &\frac{1}{4} \psi'\left(\frac{\theta^{1}+\theta^{2}}{2}\right)+\frac{1}{4}\psi'\left(\frac{\theta^{1}-\theta^{2}}{2}\right)
    \end{matrix}\right).
    \end{equation}
\end{proposition}
\begin{proof}
  Using Lemma \ref{lem:BG}, the potential function can be rewritten as
    \begin{equation}
    \phi(\theta^{1},\theta^{2})=\ln\Gamma\left(\frac{\theta^{1}-\theta^{2}}{2}\right)+\ln\Gamma\left(\frac{\theta^{1}+\theta^{2}}{2}\right)-\ln\Gamma\left(\theta^{1}\right)+(\ln2)\left(\theta^{1}-1\right).
    \end{equation}
Direct calculation then completes the proof. %    differentiation and the the definition of polygamma functions directly lead to the result. 
\end{proof}

%\begin{corollary}
    It is straightforward to compute the determinant, denoted 
    \begin{equation}\label{eq:fim}
    \det G= \frac{  \mathcal{G}}{4},
    \end{equation}
   where
       \begin{equation}
    \mathcal{G}=\psi'\left(\frac{\theta^1-\theta^2}{2}\right)\cdot\psi'\left(\frac{\theta^1+\theta^2}{2}\right)-\psi'(\theta^1)\left(\psi'\left(\frac{\theta^1-\theta^2}{2}\right)+\psi'\left(\frac{\theta^1+\theta^2}{2}\right)\right).
    \end{equation}
    Then the inverse matrix $(g^{ij})$ can be easily obtained,
    \begin{equation}
G^{-1}= \frac{1}{\mathcal{G}}\left(\begin{matrix}
    \psi'\left(\frac{\theta^{1}+\theta^{2}}{2}\right)+\psi'\left(\frac{\theta^{1}-\theta^{2}}{2}\right)
    & \psi'\left(\frac{\theta^{1}-\theta^{2}}{2}\right)-\psi'\left(\frac{\theta^{1}+\theta^{2}}{2}\right)\\
    \\
    \psi'\left(\frac{\theta^{1}-\theta^{2}}{2}\right)-\psi'\left(\frac{\theta^{1}+\theta^{2}}{2}\right)
    &\psi'\left(\frac{\theta^{1}+\theta^{2}}{2}\right)+\psi'\left(\frac{\theta^{1}-\theta^{2}}{2}\right)-4\psi'(\theta^{1}) 
    \end{matrix}\right).
    \end{equation}
%\end{corollary}

Now, by direct computation, we can obtain the connection coefficients, curvature tensor, and curvatures as follows. 
\begin{proposition}
    We have
    \begin{enumerate}
        \item the $\alpha$-connection coefficients
        \begin{equation}
        \begin{aligned}
        \Gamma_{111}^{(\alpha)} & =\frac{1-\alpha}{16}\left\{ \psi''\left(\frac{\theta^{1}+\theta^{2}}{2}\right)+\psi''\left(\frac{\theta^{1}-\theta^{2}}{2}\right)-8\psi''(\theta^{1})\right\} ,\\
        \Gamma_{121}^{(\alpha)} & =\Gamma_{112}^{(\alpha)}=\Gamma_{211}^{(\alpha)}=\Gamma_{222}^{(\alpha)}=\frac{1-\alpha}{16}\left\{ \psi''\left(\frac{\theta^{1}+\theta^{2}}{2}\right)-\psi''\left(\frac{\theta^{1}-\theta^{2}}{2}\right)\right\} ,\\
        \Gamma_{122}^{(\alpha)} & =\Gamma_{212}^{(\alpha)}=\Gamma_{221}^{(\alpha)}=\frac{1-\alpha}{16}\left\{ \psi''\left(\frac{\theta^{1}+\theta^{2}}{2}\right)+\psi''\left(\frac{\theta^{1}-\theta^{2}}{2}\right)\right\};
        \end{aligned}
        \end{equation}
        \item the nonzero $\alpha$-curvature tensor
        \begin{equation}
        \begin{aligned}
        R_{1212}^{(\alpha)}&=R_{2121}^{(\alpha)}=-R_{1221}^{(\alpha)}=-R_{2112}^{(\alpha)}\\
        &=\frac{1-\alpha^{2}}{16\mathcal{G}}\left(\psi'(\theta^{1})\psi''\left(\frac{\theta^{1}+\theta^{2}}{2}\right)\psi''\left(\frac{\theta^{1}-\theta^{2}}{2}\right)-\psi''(\theta^{1}) \right.\\
        &\quad \quad  \left.\times\left\{ \psi'\left(\frac{\theta^{1}+\theta^{2}}{2}\right)\psi''\left(\frac{\theta^{1}-\theta^{2}}{2}\right)+\psi'\left(\frac{\theta^{1}-\theta^{2}}{2}\right)\psi''\left(\frac{\theta^{1}+\theta^{2}}{2}\right)\right\} \right);
        \end{aligned}
        \end{equation}
        \item  the $\alpha$-Ricci curvature
        \begin{align*}R_{11}^{(\alpha)} & =-\frac{(1-\alpha^{2})\left(-4\psi'(\theta^{1})+\psi'\left(\frac{\theta^{1}-\theta^{2}}{2}\right)+\psi'\left(\frac{\theta^{1}+\theta^{2}}{2}\right)\right)}{16\mathcal{G}^{2}}\\
        & \quad\times\left(\psi'\left(\frac{\theta^{1}+\theta^{2}}{2}\right)\psi''(\theta^{1})\psi''\left(\frac{\theta^{1}-\theta^{2}}{2}\right)\right.+\left(\psi'\left(\frac{\theta^{1}-\theta^{2}}{2}\right)\psi''(\theta^{1})\right.\\
        & \quad\quad\quad\left.\left.-\psi'(\theta^{1})\psi''\left(\frac{\theta^{1}-\theta^{2}}{2}\right)\right)\psi''\left(\frac{\theta^{1}+\theta^{2}}{2}\right)\right), \\ 
        R_{12}^{(\alpha)}=R_{21}^{(\alpha)} & =\frac{(1-\alpha^{2})\left(\psi'\left(\frac{\theta^{1}-\theta^{2}}{2}\right)-\psi'\left(\frac{\theta^{1}+\theta^{2}}{2}\right)\right)}{16\mathcal{G}^{2}}\\
        & \quad\times\left(\psi'\left(\frac{\theta^{1}+\theta^{2}}{2}\right)\psi''\left(\theta^{1}\right)\psi''\left(\frac{\theta^{1}-\theta^{2}}{2}\right)\right.+\left(\psi'\left(\frac{\theta^{1}-\theta^{2}}{2}\right)\psi''\left(\theta^{1}\right)\right.\\
        & \quad\quad\quad\left.\left.-\psi'\left(\theta^{1}\right)\psi''\left(\frac{\theta^{1}-\theta^{2}}{2}\right)\right)\psi''\left(\frac{\theta^{1}+\theta^{2}}{2}\right)\right),\\ 
        R_{22}^{(\alpha)} & =-\frac{(1-\alpha^{2})\left(\psi'\left(\frac{\theta^{1}-\theta^{2}}{2}\right)+\psi'\left(\frac{\theta^{1}+\theta^{2}}{2}\right)\right)}{16\mathcal{G}^{2}}\\
        &\quad  \times\left(\psi'\left(\frac{\theta^{1}+\theta^{2}}{2}\right)\psi''\left(\theta^{1}\right)\psi''\left(\frac{\theta^{1}-\theta^{2}}{2}\right)\right.+\left(\psi'\left(\frac{\theta^{1}-\theta^{2}}{2}\right)\psi''\left(\theta^{1}\right)\right.\\
        & \quad\quad\quad\left.\left.-\psi'\left(\theta^{1}\right)\psi''\left(\frac{\theta^{1}-\theta^{2}}{2}\right)\right)\psi''\left(\frac{\theta^{1}+\theta^{2}}{2}\right)\right);
        \end{align*}
        \item and finally the $\alpha$-scalar curvature, expressed as $R^{(\alpha)}=8R_{1212}^{(\alpha)}/\mathcal{G}$, which is $2K^{(\alpha)}$.
    \end{enumerate}
\end{proposition}

\begin{proof}
This can be proved through a direct computation and details are omitted here.
\end{proof}

\begin{corollary}
%    \begin{enumerate}
%        \item
         Since $M$ is an exponential family, the statistical manifold is $\pm 1$-flat. 
        %Due to the factor $1-\alpha^2$, when $\alpha=\pm1$, $R^{(\alpha)}_{1212}=0$, namely, it is $\pm1$-flat and the coordinate system is $1$-affine.
     %   \item 
        When $\alpha=0$, the Gaussian curvature is 
        \begin{equation*}
        \begin{aligned}
        K & = \frac{\psi''\left(\frac{\theta^{1}+\theta^{2}}{2}\right)\left(\psi''\left(\frac{\theta^{1}-\theta^{2}}{2}\right)\psi'(\theta^{1})-\psi''(\theta^{1})\psi'\left(\frac{\theta^{1}-\theta^{2}}{2}\right)\right)}{4\left(\psi'(\theta^{1})\psi'\left(\frac{\theta^{1}-\theta^{2}}{2}\right)+\left(\psi'(\theta^{1})-\psi'\left(\frac{\theta^{1}-\theta^{2}}{2}\right)\right)\psi'\left(\frac{\theta^{1}+\theta^{2}}{2}\right)\right)^{2}}\\ & \quad -\frac{\psi''(\theta^{1})\psi''\left(\frac{\theta^{1}-\theta^{2}}{2}\right)\psi'\left(\frac{\theta^{1}+\theta^{2}}{2}\right)}{4\left(\psi'(\theta^{1})\psi'\left(\frac{\theta^{1}-\theta^{2}}{2}\right)+\left(\psi'(\theta^{1})-\psi'\left(\frac{\theta^{1}-\theta^{2}}{2}\right)\right)\psi'\left(\frac{\theta^{1}+\theta^{2}}{2}\right)\right)^{2}}.
       \end{aligned} 
        \end{equation*}
%    \end{enumerate}
\end{corollary}

\subsection{Bernoulli and Euler cases}

%These equations are difficult to be solved analytically, but recalling 
Recall that the random variable interpretation of Bernoulli and Euler umbras in \eqref{eq:RV}, both $L_B$ and $L_E$ are beta-logistic distributions with $\theta^2=0$, which make the Fisher matrix $G$ diagonal. %For simplicity, in the following, we will be focused on the submanifold including both of them, namely,  $\theta^1>0$ and $\theta^2=0$.  %example, we apply the results above to 
In particular, we emphasize the followings: 
\begin{enumerate}
    \item $L_B$, i.e., $\theta^1=2$ and $\theta^2=0$, which we will call the Bernoulli case;
    \item $L_E$, i.e., $\theta^1=1$ and $\theta^2=0$, which we will call the Euler case.
\end{enumerate}
%The re-scaling factor $t\mapsto\pi t$ in the densities of $L_B$ and $L_E$ only results in a constant factor.

\begin{remark*}
    For the Bernoulli case, $\theta^{1}=2$ and $\theta^{2}=0$, we have
    \begin{equation*}
    G_{\left(2,0\right)}=\left(\begin{matrix}\frac{1}{2}\psi'\left(1\right)-\psi'\left(2\right) & 0\\
    0 & \frac{1}{2}\psi'\left(1\right)
    \end{matrix}\right)=\left(\begin{matrix}1-\frac{\pi^{2}}{12} & 0\\
    0 & \frac{\pi^{2}}{12}
    \end{matrix}\right).
    \end{equation*}
    Therefore, all the curvatures are rational in $\pi$ and $\zeta(3)$:
    \begin{align*}
    R_{1212}^{(\alpha)} & =\frac{3\left(1-\alpha^{2}\right)\zeta\left(3\right)\left(6\zeta\left(3\right)+\pi^{2}\zeta\left(3\right)-2\pi^{2}\right)}{2\pi^{2}\left(\pi^{2}-12\right)},\\
    R_{11}^{(\alpha)} & =\frac{18\left(1-\alpha^{2}\right)\zeta\left(3\right)\left(6\zeta\left(3\right)+\pi^{2}\zeta\left(3\right)-2\pi^{2}\right)}{\pi^{4}\left(\pi^{2}-12\right)},\\
    R_{12}^{(\alpha)} & =R_{21}^{(\alpha)}=0,\\
    R_{22}^{(\alpha)} & =-\frac{18\left(1-\alpha^{2}\right)\zeta\left(3\right)\left(6\zeta\left(3\right)+\pi^{2}\zeta\left(3\right)-2\pi^{2}\right)}{\pi^{2}\left(\pi^{2}-12\right)^{2}},
    \end{align*}
    and 
    \begin{equation*}
    R^{(\alpha)}=-\frac{432\left(1-\alpha^{2}\right)\zeta\left(3\right)\left(6\zeta\left(3\right)+\pi^{2}\zeta\left(3\right)-2\pi^{2}\right)}{\pi^{4}\left(\pi^{2}-12\right)^{2}}.
    \end{equation*}
    For the Euler case, $\theta^{1}=1$ and $\theta^{2}=0$, the Fisher information metric is
    \begin{equation*}
    G_{\left(1,0\right)}=\left(\begin{matrix}\frac{\pi^{2}}{12} & 0\\
    0 & \frac{\pi^{2}}{4}
    \end{matrix}\right),
    \end{equation*}
   and the curvatures are
    \begin{align*}
    R_{1212}^{(\alpha)} & =\frac{7\left(1-\alpha^{2}\right)\zeta\left(3\right)^{2}}{2\pi^{2}},\\
    R_{11}^{(\alpha)} & =\frac{14\left(1-\alpha^{2}\right)\zeta\left(3\right)^{2}}{\pi^{4}},\quad R_{12}^{(\alpha)} =R_{21}^{(\alpha)}=0,\quad R_{22}^{(\alpha)}  =\frac{42\left(1-\alpha^{2}\right)\zeta\left(3\right)^{2}}{\pi^{4}},\\
    R^{(\alpha)}&=\frac{336\left(1-\alpha^{2}\right)\zeta\left(3\right)^{2}}{\pi^{6}}.
    \end{align*}
    \end{remark*}
  
  \subsection{Geodesics and their stability}
Finally, we obtain the geodesics (of the Riemannian case) and analyze their stability.
\begin{proposition}\label{prop:geodesic}
    The geodesic equations are given by
            \begin{align*}
       & \frac{\operatorname{d}^{2}\!\theta^{1}}{\operatorname{d}\!t^{2}}-
        \left(\frac{\psi'\left(\frac{\theta^{1}+\theta^{2}}{2}\right)\left(4\psi''(\theta^{1})-\psi''\left(\frac{\theta^{1}-\theta^{2}}{2}\right)\right)}{8\mathcal{G}}\right.\\
        &\quad\quad \quad\quad  \quad\quad   + \left. \frac{\psi'\left(\frac{\theta^{1}-\theta^{2}}{2}\right)\left(4\psi''(\theta^{1})-\psi''\left(\frac{\theta^{1}+\theta^{2}}{2}\right)\right)}{8\mathcal{G}}\right)\left(\frac{\operatorname{d}\!\theta^{1}}{\operatorname{d}\!t}\right)^{2}\\
        &+\frac{\psi''\left(\frac{\theta^{1}+\theta^{2}}{2}\right)\psi'\left(\frac{\theta^{1}-\theta^{2}}{2}\right)-\psi''\left(\frac{\theta^{1}-\theta^{2}}{2}\right)\psi'\left(\frac{\theta^{1}+\theta^{2}}{2}\right)}{4\mathcal{G}}\frac{\operatorname{d}\!\theta^{1}}{\operatorname{d}\!t}\frac{\operatorname{d}\!\theta^{2}}{\operatorname{d}\!t}\\
        &+\frac{\psi''\left(\frac{\theta^{1}-\theta^{2}}{2}\right)\psi'\left(\frac{\theta^{1}+\theta^{2}}{2}\right)+\psi''\left(\frac{\theta^{1}+\theta^{2}}{2}\right)\psi'\left(\frac{\theta^{1}-\theta^{2}}{2}\right)}{8\mathcal{G}}\left(\frac{\operatorname{d}\!\theta^{2}}{\operatorname{d}\!t}\right)^{2}=0
    \end{align*}
    and 
    \begin{align*}
        &\frac{\operatorname{d}^{2}\!\theta^{2}}{\operatorname{d}\!t^{2}}+\left(\frac{\psi'\left(\frac{\theta^{1}+\theta^{2}}{2}\right)\left(4\psi''(\theta^{1})-\psi''\left(\frac{\theta^{1}-\theta^{2}}{2}\right)\right)}{8\mathcal{G}}\right.+\frac{\left(\psi''\left(\frac{\theta^{1}-\theta^{2}}{2}\right)-\psi''\left(\frac{\theta^{1}+\theta^{2}}{2}\right)\right)\psi'(\theta^{1})}{4\mathcal{G}}\\
        &\quad \quad \quad\quad  \quad\quad    +\left.\frac{\left(\psi''\left(\frac{\theta^{1}+\theta^{2}}{2}\right)-4\psi''(\theta^{1})\right)\psi'\left(\frac{\theta^{1}-\theta^{2}}{2}\right)}{8\mathcal{G}}\right)\left(\frac{\operatorname{d}\!\theta^{1}}{\operatorname{d}\!t}\right)^{2} \allowdisplaybreaks\\
        &+ \frac{\psi''\left(\frac{\theta^{1}-\theta^{2}}{2}\right)\left(\psi'\left(\frac{\theta^{1}+\theta^{2}}{2}\right)-2\psi'(\theta^{1})\right)
        +\psi''\left(\frac{\theta^{1}+\theta^{2}}{2}\right)\left(\psi'\left(\frac{\theta^{1}-\theta^{2}}{2}\right)-2\psi'(\theta^{1})\right)}{4\mathcal{G}} \frac{\operatorname{d}\!\theta^{1}}{\operatorname{d}\!t}\frac{\operatorname{d}\!\theta^{2}}{\operatorname{d}\!t}\\
        &+\frac{\psi''\left(\frac{\theta^{1}-\theta^{2}}{2}\right)\left(2\psi'(\theta^{1})-\psi'\left(\frac{\theta^{1}+\theta^{2}}{2}\right)\right)
        +\psi''\left(\frac{\theta^{1}+\theta^{2}}{2}\right)\left(\psi'\left(\frac{\theta^{1}-\theta^{2}}{2}\right)-2\psi'(\theta^{1})\right)}{8\mathcal{G}}\left(\frac{\operatorname{d}\!\theta^{2}}{\operatorname{d}\!t}\right)^{2}=0.
    \end{align*}
\end{proposition}

The geodesic equations cannot be easily solved analytically, and here we seek for its numerical solutions by using the ode45 of Matlab in the domain $\theta^1\pm\theta^2>0$. Fig.~\ref{fig:geo} shows a variety of  geodesics starting from $(1,0)$. 

The relative geodesic spread can be characterised by the Jacobi field, which satisfies the
Jacobi--Levi-Civita equations (see, e.g., \cite{doC1992,Pengetal2011}). Fig.~\ref{fig:geo}  implies that the geodesic spread is at least linearly unstable. 

\begin{figure}[htbp]
\centering
\includegraphics[scale=0.3]{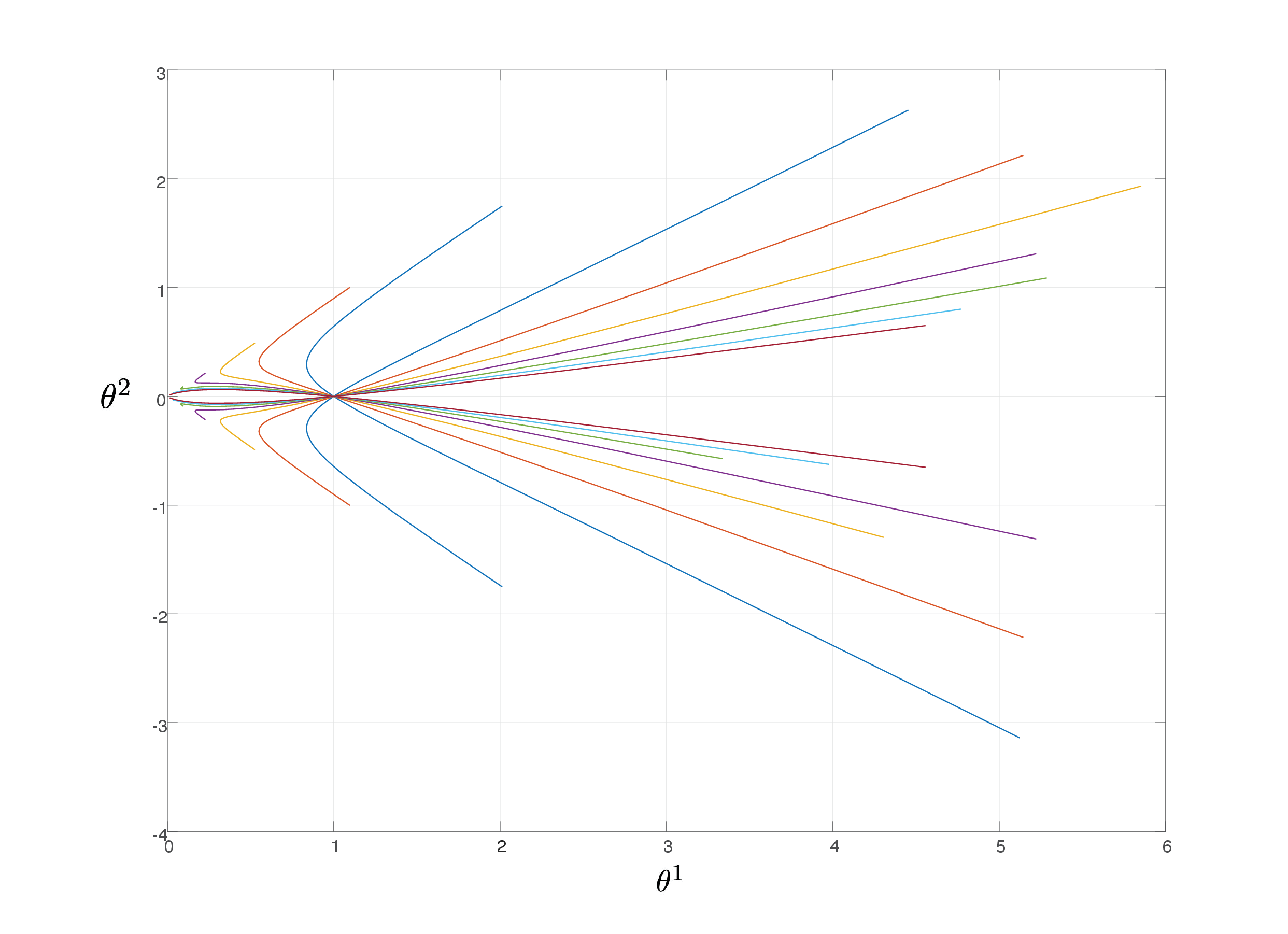}
\caption{Geodesics starting from  $(1,0)$ in the domain $\theta^1\pm\theta^2>0$. }
\label{fig:geo}
\end{figure}

%The corresponding Jacobi--Levi-Civita equations \eqref{eq:JLC} can
%%\begin{equation}
%%xx,
%%\end{equation}
%similarly be solved numerically using ode45 of Matlab, shown in Fig. xxx. Numerically, we can also conclude that the geodesic spread is described approximately as
%\begin{equation}
%J^2\approx \exp(xxxxxt).
%\end{equation}
%
%
%
%
%{\color{red} To Jiongyi by LP: To analyze the stability of geodesics, Jacobi fields (satisfying the Jacobi equations) should be used, c.f. \cite{Pengetal2011}. As they cannot be solved analytically, numerical methods should be used.}

\section{$\alpha$-parallel prior}

Takeuchi and Amari \cite{alphaParallel} studies the invariant $\alpha$-parallel priors with respect to statistical manifolds. Our statistical manifold $M$ given by \eqref{eq:stma} is $\pm 1$-flat, and is hence statistically equiaffine, i.e., 
\begin{equation}
\partial_i T_{jk}^k=\partial_jT_{ik}^k,\quad i,j=1,2,
\end{equation}
holds everywhere in $M$, where $T_{ij}^k=T_{ijl}g^{lk}$.
Consequently, $M$ admits $\alpha$-parallel prior $\omega^{(\alpha)}(\theta)$ for any $\alpha\in\mathbb{R}$, satisfying 
\begin{equation}
\partial_i\omega^{(\alpha)}(\theta) = \Gamma_{ij}^{(\alpha)j} \omega^{(\alpha)}(\theta).
\end{equation}
For the exponential family $M$, this can be simply solved, yielding 
\begin{equation}\label{eq:om}
\omega^{(\alpha)}(\theta)  = \left(\det G\right)^{(1-\alpha)/2},
\end{equation}
where $G$ is the Fisher matrix \eqref{eq:fis}.
In particular, Jeffreys prior (e.g., \cite{Jef1998}) is the $0$-prior, while the left-invariant Haar measure (e.g., \cite{Zidek1969}) is the $-1$-prior.

%$\det G=\mathcal{G}/4$
%       \begin{equation}
%    \mathcal{G}=\psi'\left(\frac{\theta^1-\theta^2}{2}\right)\cdot\psi'\left(\frac{\theta^1+\theta^2}{2}\right)-\psi'(\theta^1)\left(\psi'\left(\frac{\theta^1-\theta^2}{2}\right)+\psi'\left(\frac{\theta^1+\theta^2}{2}\right)\right).
%    \end{equation}

For a series of independent observations $\mathbf{x}=\{x_1,x_2,\ldots,x_N\}$ from the beta-logistic distribution \eqref{eq:pdf} with unknown parameters $\theta^1,\theta^2$, the posterior distribution $p_{\alpha}(\theta^1,\theta^2 \mid \mathbf{x}) $ can be obtained by employing \eqref{eq:om} as the improper prior. From Bayes' theorem, we have 
\begin{equation}\label{eq:pd}
p_{\alpha}(\theta^1,\theta^2 \mid \mathbf{x}) = \frac{p(\mathbf{x};\theta^1,\theta^2) \left(\det G\right)^{(1-\alpha)/2}}{A(\mathbf{x})},
\end{equation}
where
\begin{equation}
A(\mathbf{x})=\int_0^{\infty} \int_{-\theta^1}^{\theta^1} p(\mathbf{x};\theta^1,\theta^2) \left(\det G\right)^{(1-\alpha)/2} \operatorname{d}\!\theta^2\operatorname{d}\!\theta^1.
\end{equation}
The joint distribution of the observations $\mathbf{x}$ is
\begin{equation}
\begin{aligned}
p(\mathbf{x};\theta^1,\theta^2) =\prod_{i=1}^N p(x_i;\theta^1,\theta^2)
= \frac{2^{N(1-\theta^1)}\left(a(\mathbf{x})\right)^{\theta^1} \exp(\theta^2 b(\mathbf{x}))}{B^N\left(\frac{\theta^1-\theta^2}{2},\frac{\theta^1+\theta^2}{2}\right)},
\end{aligned}
\end{equation}
where
\begin{equation}
a(\mathbf{x})=\prod\limits_{i=1}^N\sech(x_i),\quad b(\mathbf{x})=\sum\limits_{i=1}^Nx_i. 
\end{equation}
The maximum a posteriori estimation for the parameters $\theta=(\theta^1,\theta^2)$ can hence be obtained by
\begin{equation}
\widehat{\theta} =\underset{\theta}{\arg \max} ~~ p_{\alpha}(\theta^1,\theta^2 \mid \mathbf{x}).
\end{equation}
%\rule[0.5ex]{1\columnwidth}{1pt}
Introduce natural logarithm of the posterior distribution \eqref{eq:pd}
\begin{equation}
\begin{aligned}
l_{\alpha}(\theta^1,\theta^2\mid \mathbf{x})& =\log p_{\alpha}(\theta^1,\theta^2\mid \mathbf{x})\\
&= N(1-\theta^1)\log 2+ \theta^{1}\log a(\mathbf{x})+\theta^{2}b(\mathbf{x})   -N\log B\left(\frac{\theta^1-\theta^2}{2},\frac{\theta^1+\theta^2}{2}\right) \\ & \quad +\frac{1-\alpha}{2}\log \det G - \log A(\mathbf{x}),
\end{aligned}
\end{equation}
where determinant of the Fisher  matrix is (see eq. \eqref{eq:fim})
\begin{equation}
\det G=4\left(\psi'\left(\frac{\theta^{1}-\theta^{2}}{2}\right)\cdot\psi'\left(\frac{\theta^{1}+\theta^{2}}{2}\right)-\psi'(\theta^{1})\left(\psi'\left(\frac{\theta^{1}-\theta^{2}}{2}\right)+\psi'\left(\frac{\theta^{1}+\theta^{2}}{2}\right)\right)\right),
\end{equation}
and the maximum a posteriori estimation $\widehat{\theta}$ can be determined by $\nabla l_{\alpha}=0$, where $\nabla$ denotes the gradient with respect to $\theta$. These equations can be obtained by direct calculation, which read
\begin{equation}\label{eq:maxipos}
\begin{aligned}
\log \frac{a(\mathbf{x})}{2^N}-\frac{N}{2}\left(\psi\left(\frac{\theta^1+\theta^2}{2}\right)+\psi\left(\frac{\theta^1-\theta^2}{2}\right)-2\psi(\theta^1)  \right)+\frac{1-\alpha}{2}\frac{\partial_{\theta^1}\!\det G}{\det G} &=0,\\
b(\mathbf{x})-\frac{N}{2}\left(\psi\left(\frac{\theta^1+\theta^2}{2}\right)-\psi\left(\frac{\theta^1-\theta^2}{2}\right)\right)+\frac{1-\alpha}{2}\frac{\partial_{\theta^2}\!\det G}{\det G}&=0.
\end{aligned}
\end{equation}
Here, we used the properties 
\begin{equation}
B(x,y)=B(y,x) \text{ and } \partial_xB(x,y)=B(x,y)(\psi(x)-\psi(x+y)).
\end{equation}
Solving the system \eqref{eq:maxipos}  analytically proves to be challenging. Numerical methods can be utilised in practical applications, such as Newton's method, gradient descent, and fixed-point iteration.

\section{Conclusion}\label{sec:conclustion}
Inspired by the probabilistic interpretation of Bernoulli and Euler polynomials as moments of certain distributions, we study the exponential generalized beta of the second distribution, which includes both Bernoulli and Euler cases as special evaluations. Besides the geometric structure, e.g., Fisher information metric, $\alpha$-connection coefficients, $\alpha$-curvature tensors, $\alpha$-Ricci curvature, and the $\alpha$-scalar curvature, we also provide the geodesic equations, which imply that the geodesic spread is at least linearly unstable. Finally, as the statistical manifold is $\pm 1$-flat and hence equiaffine, we give the $\alpha$-parallel prior and the equations determining the maximum a posteriori estimation. Some future work include analytic and numerical solutions to those differential equations, though the former seems to be challenging.

{\bf Acknowledgement.} L. Peng was partially supported by JSPS KAKENHI grant number JP20K14365, JST-CREST grant number JPMJCR1914, the Fukuzawa Fund and KLL of Keio University.

\end{document}